\documentclass[11pt,a4paper]{article}

\usepackage{inputenc}
\usepackage{amsmath}
\usepackage{bm}
\usepackage{bbold}
\usepackage{amsthm}
\usepackage{enumerate}

\usepackage{hyperref}

\title{Methods of Tropical Optimization in Rating Alternatives Based on Pairwise Comparisons}

\author{N. Krivulin\thanks{St.~Petersburg State University, Universitetskaya nab.~7/9, St.~Petersburg, 199034, Russia, nkk@math.spbu.ru.}}

\date{}

\newtheorem{theorem}{Theorem}
\newtheorem{lemma}[theorem]{Lemma}

\theoremstyle{definition}
\newtheorem{example}{Example}

\begin{document}

\maketitle

\begin{abstract}
We apply methods of tropical optimization to handle problems of rating alternatives on the basis of the log-Chebyshev approximation of pairwise comparison matrices. We derive a direct solution in a closed form, and investigate the obtained solution when it is not unique. Provided the approximation problem yields a set of score vectors, rather than a unique (up to a constant factor) one, we find those vectors in the set, which least and most differentiate between the alternatives with the highest and lowest scores, and thus can be representative of the entire solution.
\\

\\

\end{abstract}

\section{Introduction}

Tropical (idempotent) mathematics, which deals with the theory and applications of semirings with idempotent addition \cite{Golan2003Semirings,Heidergott2006Maxplus}, finds application in operations research, computer science and other fields. Optimization problems that are formulated and solved in the framework of tropical mathematics constitute an important research domain, which offers new solutions to old and novel problems in various applied areas, including project scheduling \cite{Krivulin2016Maximization,Krivulin2015Extremal}, location analysis \cite{Krivulin2015Soving} and decision making \cite{Krivulin2015Rating,Krivulin2016Using}. The problems are usually defined to minimize or maximize functions on vectors over idempotent semifields (semirings with multiplicative inverses).

In this paper, we apply methods of tropical optimization to handle problems of rating alternatives on the basis of the log-Chebyshev approximation of pairwise comparison matrices. We derive a direct solution in a closed form, and investigate the solution when it is not unique. Provided the approximation problem yields a set of score vectors, rather than a unique (up to a constant factor) one, we find those vectors in the set, which least and most differentiate between the alternatives with the highest and lowest scores, and thus can be representative of the entire solution.

\section{Rating Alternatives via Pairwise Comparisons}

The method of rating alternatives from pairwise comparisons finds use in decision making when a direct evaluation of the ratings is unacceptable or infeasible (see, e.g., \cite{Saaty2013Onthemeasurement} for further details). The outcome of the comparisons is described by a square symmetrically reciprocal matrix $\bm{A}=(a_{ij})$, where $a_{ij}$ shows the relative preference of alternative $i$ over $j$, and satisfies the condition $a_{ij}=1/a_{ji}>0$ for all $i,j$.

To provide consistency of the data given by pairwise comparison matrices, the entries of the matrices must be transitive to provide the equality $a_{ij}=a_{ik}a_{kj}$ for all $i,j,k$. A pairwise comparison matrix with only transitive entries is called consistent.

For each consistent matrix $\bm{A}=(a_{ij})$, there is a positive vector $\bm{x}=(x_{i})$ whose elements completely determine the entries of $\bm{A}$ by the relation $a_{ij}=x_{i}/x_{j}$. Provided that a matrix $\bm{A}$ is consistent, its corresponding vector $\bm{x}$ is considered to represent directly, up to a positive factor, the individual scores of alternatives in question.

The pairwise comparison matrices encountered in practice are generally inconsistent, which leads to a problem of approximating these matrices by consistent matrices. To solve the problem, the approximation with the principal eigenvector \cite{Saaty1984Comparison,Saaty2013Onthemeasurement}, least squares approximation \cite{Saaty1984Comparison,Chu1998Ontheoptimal} and other techniques \cite{Barzilai1997Deriving,Farkas2003Consistency,Gonzalezpachon2003Transitive} are used.

Another approach involves the approximation of a reciprocal matrix $\bm{A}=(a_{ij})$ by a consistent matrix $\bm{X}=(x_{ij})$ in the log-Chebyshev sense, where the approximation error is measured with the Chebyshev metric on the logarithmic scale. Since both matrices $\bm{A}$ and $\bm{X}$ have positive entries, and the logarithm is monotone increasing, the error can be written as $\max_{i,j}|\log a_{ij}-\log x_{ij}|=\log\max_{i,j}\max\{a_{ij}/x_{ij},x_{ij}/a_{ij}\}$.

Considering that the minimization of the logarithm is equivalent to minimizing its argument, and that the matrix $\bm{X}$ can be defined through a positive vector $\bm{x}=(x_{i})$ by the equality $x_{ij}=x_{i}/x_{j}$ for all $i,j$, the error function to minimize is replaced by $\max_{i,j}\max\{a_{ij}/x_{ij},x_{ij}/a_{ij}\}=\max_{i,j}\max\{a_{ij}x_{j}/x_{i},a_{ji}x_{i}/x_{j}\}$. The application of the condition $a_{ij}=1/a_{ji}$ yields $\max_{i,j}\max\{a_{ij}x_{j}/x_{i},a_{ji}x_{i}/x_{j}\}=\max_{i,j}a_{ij}x_{j}/x_{i}$, which finally reduces the approximation problem to finding positive vectors $\bm{x}$ to
\begin{equation}
\begin{aligned}
&
\text{minimize}
&&
%\textstyle{\max_{i,j}a_{ij}x_{j}/x_{i}}.
\max_{i,j}a_{ij}x_{j}/x_{i}.
\end{aligned}
\label{P-minaijxjxi}
\end{equation}

Assume that the approximation results in a set $\mathcal{S}$ of score vectors $\bm{x}$, rather than a unique (up to a constant factor) one. Then, further analysis is needed to reduce to a very few representative solutions, such as some ``worst'' and ``best'' solutions.

As the purpose of calculating the scores is to differentiate alternatives, one can concentrate on two vectors $\bm{x}=(x_{i})$ from $\mathcal{S}$, which least and most differentiate between the alternatives with the highest and lowest scores by minimizing and maximizing the contrast ratio $\max_{i}x_{i}/\min_{i}x_{i}=\max_{i}x_{i}\cdot\max_{i}x_{i}^{-1}$. Then, the problem of calculating the least (the most) differentiating solution is to find vectors $\bm{x}\in\mathcal{S}$ that
\begin{equation}
\begin{aligned}
&
\text{minimize (maximize)}
&&
%\textstyle{\max_{i}x_{i}\cdot\max_{i}x_{i}^{-1}}.
\max_{i}x_{i}\cdot\max_{i}x_{i}^{-1}.
\end{aligned}
\label{P-minmaxxixi}
\end{equation}

Below, we reformulate problems \eqref{P-minaijxjxi} and \eqref{P-minmaxxixi} in terms of tropical mathematics, and then apply recent results in tropical optimization to offer complete, direct solutions.

\section{Preliminary Definitions, Notation and Results}
\label{S-PDNR}

We start with a brief overview of the basic definitions and notation of tropical algebra. For further details on tropical mathematics, see, eg, recent publications \cite{Golan2003Semirings,Heidergott2006Maxplus}. 

Consider the set of nonnegative reals $\mathbb{R}_{+}$, which is equipped with two operations, addition $\oplus$ defined as maximum, and multiplication $\otimes$ defined as usual, and has $0$ and $1$ as their neutral elements. Addition is idempotent, since $x\oplus x=\max(x,x)=x$ for all $x\in\mathbb{R}_{+}$. Multiplication is distributive over addition and invertible to give each $x\ne0$ an inverse $x^{-1}$ such that $x\otimes x^{-1}=xx^{-1}=1$. The system $(\mathbb{R}_{+},\oplus,\otimes,0,1)$ is called the idempotent semifield or the max-algebra and denoted $\mathbb{R}_{\max}$. In the sequel, the sign $\otimes$ is omitted for brevity. The power notation has the standard meaning.

The set of matrices over $\mathbb{R}_{+}$ with $m$ rows and $n$ columns is denoted by $\mathbb{R}_{+}^{m\times n}$. A matrix with all zero entries is the zero matrix. A matrix without zero rows is called row-regular. Matrix operations employ the conventional entry-wise formulae, where the scalar operations $\oplus$ and $\otimes$ play the role of the usual addition and multiplication.

The multiplicative conjugate transpose of a nonzero matrix $\bm{A}=(a_{ij})$ is the matrix $\bm{A}^{-}=(a_{ij}^{-})$ with the entries $a_{ij}^{-}=a_{ji}^{-1}$ if $a_{ji}\ne0$, and $a_{ij}^{-}=0$ otherwise.

Consider the square matrices in the set $\mathbb{R}_{+}^{n\times n}$. A matrix with $1$ along the diagonal and $0$ elsewhere is the identity matrix denoted $\bm{I}$. The power notation specifies iterated products as $\bm{A}^{0}=\bm{I}$ and $\bm{A}^{p}=\bm{A}^{p-1}\bm{A}$ for any matrix $\bm{A}$ and integer $p>0$.

The tropical spectral radius of a matrix $\bm{A}=(a_{ij})\in\mathbb{R}_{+}^{n\times n}$ is the scalar given by
\begin{equation}
\lambda
=
\bigoplus_{1\leq k\leq n}\bigoplus_{1\leq i_{1},\ldots,i_{k}\leq n}(a_{i_{1}i_{2}}a_{i_{2}i_{3}}\cdots a_{i_{k}i_{1}})^{1/k}.
\label{E-lambda-ai1i2ai2i3aiki1}
\end{equation}

The asterate operator (the Kleene star) maps the matrix $\bm{A}$ onto the matrix
\begin{equation}
\bm{A}^{\ast}
=
\bm{I}\oplus\bm{A}\oplus\cdots\oplus\bm{A}^{n-1}.
\label{E-Aast}
\end{equation}

The column vectors with $n$ elements form the set $\mathbb{R}_{+}^{n}$. The vectors with all elements equal to $0$ and to $1$ are denoted by $\bm{0}$ and $\bm{1}$. A vector is regular if it has no zero elements. For any nonzero column vector $\bm{x}=(x_{i})$, its conjugate transpose is the row vector $\bm{x}^{-}=(x_{i}^{-})$, where $x_{i}^{-}=x_{i}^{-1}$ if $x_{i}\ne0$, and $x_{i}^{-}=0$ otherwise. 

We conclude the overview with examples of tropical optimization problems. Suppose that, given a matrix $\bm{A}=(a_{ij})\in\mathbb{R}_{+}^{n\times n}$, we need to find vectors $\bm{x}\in\mathbb{R}_{+}^{n}$ that
\begin{equation}
\begin{aligned}
&
\text{minimize}
&&
\bm{x}^{-}\bm{A}\bm{x}.
\end{aligned}
\label{P-minxAx}
\end{equation}

The next complete, direct solution to the problem is obtained in \cite{Krivulin2015Extremal}.
\begin{lemma}
\label{L-minxAx}
Let $\bm{A}$ be a matrix with spectral radius $\lambda>0$. Then, the minimum value in \eqref{P-minxAx} is equal to $\lambda$, and all regular solutions are given by
\begin{equation*}
\bm{x}
=
(\lambda^{-1}\bm{A})^{\ast}\bm{u},
\qquad
\bm{u}\ne\bm{0}.
\end{equation*}
\end{lemma}

Given a matrix $\bm{A}\in\mathbb{R}_{+}^{m\times n}$ and vectors $\bm{p}\in\mathbb{R}_{+}^{m}$, $\bm{q}\in\mathbb{R}_{+}^{n}$, we now find $\bm{x}\in\mathbb{R}_{+}^{n}$ that
\begin{equation}
\begin{aligned}
&
\text{minimize}
&&
\bm{q}^{-}\bm{x}(\bm{A}\bm{x})^{-}\bm{p}.
\end{aligned}
\label{P-minqxAxp}
\end{equation}

A solution given by \cite{Krivulin2015Soving} uses a sparsification technique to provide the next result.
\begin{lemma}
\label{L-minqxAxp}
Let $\bm{A}=(a_{ij})$ be a row-regular matrix, $\bm{p}=(p_{i})$ be nonzero and $\bm{q}=(q_{j})$ be regular vectors, and $\Delta=(\bm{A}\bm{q})^{-}\bm{p}$. Let $\widehat{\bm{A}}=(\widehat{a}_{ij})$ denote the matrix with entries
\begin{equation*}
\widehat{a}_{ij}
=
\begin{cases}
a_{ij}, & \text{if $a_{ij}\geq\Delta^{-1}p_{i}q_{j}^{-1}$};
\\
0, & \text{otherwise}.
\end{cases}
\end{equation*}

Let $\mathcal{A}$ be the set of matrices obtained from $\widehat{\bm{A}}$ by fixing one nonzero entry in each row and setting the others to $0$.

Then, the minimum value in problem \eqref{P-minqxAxp} is equal to $\Delta=(\bm{A}\bm{q})^{-}\bm{p}$, and all regular solutions are given by the conditions
\begin{equation*}
\bm{x}
=
(\bm{I}\oplus\Delta^{-1}\bm{A}_{1}^{-}\bm{p}\bm{q}^{-})\bm{u},
\qquad
\bm{u}\ne\bm{0},
\qquad
\bm{A}_{1}\in\mathcal{A}.
\end{equation*}
\end{lemma}

Finally, we consider a maximization version of problem \eqref{P-minqxAxp} to find vectors $\bm{x}$ that
\begin{equation}
\begin{aligned}
&
\text{maximize}
&&
\bm{q}^{-}\bm{x}(\bm{A}\bm{x})^{-}\bm{p}.
\end{aligned}
\label{P-maxqxAxp}
\end{equation}

A complete solution to the problem is obtained in \cite{Krivulin2016Maximization}. Below, we describe this solution in a more compact vector form using the representation lemma in \cite{Krivulin2015Soving}.
\begin{lemma}
\label{L-maxqxAxp}
Let $\bm{A}=(\bm{a}_{j})$ be a matrix with regular columns $\bm{a}_{j}=(a_{ij})$, and $\bm{p}=(p_{i})$ and $\bm{q}=(q_{j})$ be regular vectors. Let $\bm{A}_{sk}$ denote the matrix obtained from $\bm{A}$ by fixing the entry $a_{sk}$ for some indices $s$ and $k$, and replacing the other entries by $0$.

Then, the maximum value in \eqref{P-maxqxAxp} is equal to $\Delta=\bm{q}^{-}\bm{A}^{-}\bm{p}$, and all regular solutions are given by
\begin{equation*}
\bm{x}
=
(\bm{I}\oplus\bm{A}_{sk}^{-}\bm{A})\bm{u},
\quad
\bm{u}\ne\bm{0},
\quad
k
=
\arg\max_{j}q_{j}^{-1}\bm{a}_{j}^{-}\bm{p},
\quad
s
=
\arg\max_{i}a_{ik}^{-1}p_{i}.
\end{equation*}
\end{lemma}

\section{Application to Rating Alternatives}

We are now in a position to represent optimization problems \eqref{P-minaijxjxi} and \eqref{P-minmaxxixi} stated above in the tropical mathematics setting, and then to solve them in an explicit form.

Consider problem \eqref{P-minaijxjxi} of evaluating the score vector based on the log-Chebyshev approximation of a pairwise comparison matrix $\bm{A}$. In terms of the max-algebra $\mathbb{R}_{\max}$ the problem takes the form \eqref{P-minxAx}. Application of Lemma~\ref{L-minxAx} yields the following result.
\begin{theorem}
\label{T-minxAx}
Let $\bm{A}$ be a pairwise comparison matrix with spectral radius $\lambda$, and denote $\bm{A}_{\lambda}=\lambda^{-1}\bm{A}$ and $\bm{B}=\bm{A}_{\lambda}^{\ast}$. Then, all score vectors are given by
\begin{equation*}
\bm{x}
=
\bm{B}\bm{u},
\qquad
\bm{u}\ne\bm{0}.
\end{equation*}
\end{theorem}
\begin{example}\label{X-A}
Suppose the result of comparing $n=4$ alternatives is given by the matrix
\begin{equation}
\bm{A}=\left(
\begin{matrix}
1 & 1/3 & 1/2 & 1/3
\\
3 & 1 & 4 & 1
\\
2 & 1/4 & 1 & 2
\\
3 & 1 & 1/2 & 1
\end{matrix}
\right).
\label{E-A}
\end{equation}

To apply Theorem~\ref{T-minxAx}, we use \eqref{E-lambda-ai1i2ai2i3aiki1} to find $\lambda=(a_{23}a_{34}a_{42})^{1/3}=2$, and calculate
\begin{equation*}
\bm{A}_{\lambda}
=
\left(
\begin{matrix}
1/2 & 1/6 & 1/4 & 1/6
\\
3/2 & 1/2 & 2 & 1/2
\\
1 & 1/8 & 1/2 & 1
\\
3/2 & 1/2 & 1/4 & 1/2
\end{matrix}
\right).
\end{equation*}

Then, we follow \eqref{E-Aast} to compute
\begin{equation*}
\bm{A}_{\lambda}^{\ast}
=
\left(
\begin{matrix}
1 & 1/6 & 1/3 & 1/3
\\
3 & 1 & 2 & 2
\\
3/2 & 1/2 & 1 & 1
\\
3/2 & 1/2 & 1 & 1
\end{matrix}
\right).
\end{equation*}

As the last three columns of the matrix $\bm{A}_{\lambda}^{\ast}$ are collinear, we take one of them, say, the second. Combining with the first column multiplied by $1/3$ leads to the solution
\begin{equation}
\bm{x}
=
\bm{B}\bm{u},
\qquad
\bm{B}
=
\left(
\begin{matrix}
1/3 & 1/6
\\
1 & 1
\\
1/2 & 1/2
\\
1/2 & 1/2
\end{matrix}
\right),
\qquad
\bm{u}
=
(u_{1},u_{2})^{T},
\qquad
u_{1},u_{2}
\ne
0.
\label{E-xBu}
\end{equation}

Note that all the solutions assign the highest score to the second alternative and the lowest to the first. Moreover, the solutions which least and most differentiate between these alternatives, are the first and the second columns in the matrix $\bm{B}$. 
\end{example}

In the general case, the least and most differentiating solutions from a set of vectors, given in the form $\bm{x}=\bm{B}\bm{u}$, are determined by solving problems \eqref{P-minmaxxixi}. The problems are to minimize and maximize the contrast ratio for the elements of the vector $\bm{x}$, which, in terms of tropical mathematics, takes the form $\bm{1}^{T}\bm{x}\bm{x}^{-}\bm{1}=\bm{1}^{T}\bm{B}\bm{u}(\bm{B}\bm{u})^{-}\bm{1}$. 

To find a vector $\bm{x}=\bm{B}\bm{u}$ with the least differentiation between scores, we solve the problem
\begin{equation*}
\begin{aligned}
&
\text{minimize}
&&
\bm{1}^{T}\bm{B}\bm{u}(\bm{B}\bm{u})^{-}\bm{1}.
\end{aligned}
%\label{P-min1BuBu1}
\end{equation*}

Assuming the matrix $\bm{B}$ is obtained as in Theorem~\ref{T-minxAx}, we have the next result.
\begin{theorem}\label{T-min1BuBu1}
Let $\widehat{\bm{B}}$ be a sparsified matrix derived from $\bm{B}$ by setting to $0$ all entries below $\Delta^{-1}=((\bm{B}(\bm{1}^{T}\bm{B})^{-})^{-}\bm{1})^{-1}$, and $\mathcal{B}$ be the set of matrices obtained from $\widehat{\bm{B}}$ by fixing one nonzero entry in each row and setting the others to $0$. Then, the least differentiating score vectors are given by
\begin{equation*}
\bm{x}
=
\bm{B}(\bm{I}\oplus\Delta^{-1}\bm{B}_{1}^{-}\bm{1}\bm{1}^{T}\bm{B})\bm{v},
\qquad
\bm{v}\ne\bm{0},
\qquad
\bm{B}_{1}\in\mathcal{B}.
\end{equation*}
\end{theorem}
\begin{proof}
We reduce the problem under study to \eqref{P-minqxAxp} by the substitutions $\bm{q}^{-}=\bm{1}^{T}\bm{B}$, $\bm{A}=\bm{B}$, $\bm{p}=\bm{1}$ and $\bm{x}=\bm{u}$. Since the matrix $\bm{B}$ has only nonzero entries, the regularity conditions of Lemma~\ref{L-minqxAxp} are satisfied. Application of this lemma involves evaluating the minimum value $\Delta=(\bm{B}(\bm{1}^{T}\bm{B})^{-})^{-}\bm{1}$, calculating the sparsified matrix $\widehat{\bm{B}}$, and forming the matrix set $\mathcal{B}$. The solution is given by $\bm{u}=(\bm{I}\oplus\Delta^{-1}\bm{B}_{1}^{-}\bm{1}\bm{1}^{T}\bm{B})\bm{v}$, where $\bm{v}\ne\bm{0}$ and $\bm{B}_{1}\in\mathcal{B}$. Turning back to the vector $\bm{x}=\bm{B}\bm{u}$ yields the desired result.
%\qed
\end{proof}
\begin{example}
Consider the solution obtained in the form \eqref{E-xBu} in Example~\ref{X-A} for the matrix \eqref{E-A}. To apply the result of Theorem~\ref{T-min1BuBu1}, we successively calculate
\begin{equation*}
\bm{1}^{T}\bm{B}
=
\left(
\begin{array}{cc}
1 & 1
\end{array}
\right),
\qquad
\bm{B}(\bm{1}^{T}\bm{B})^{-}
=
\left(
\begin{matrix}
1/3
\\
1
\\
1/2
\\
1/2
\end{matrix}
\right),
\qquad
\Delta
=
(\bm{B}(\bm{1}^{T}\bm{B})^{-})^{-}\bm{1}
=
3,
\end{equation*}
and
\begin{equation*}
\widehat{\bm{B}}
=
\left(
\begin{matrix}
1/3 & 0
\\
1 & 1
\\
1/2 & 1/2
\\
1/2 & 1/2
\end{matrix}
\right).
\end{equation*}

We now examine the matrices obtained from $\widehat{\bm{B}}$ by leaving one nonzero entry in each row. For instance, consider the matrix
\begin{equation*}
\bm{B}_{1}
=
\left(
\begin{matrix}
1/3 & 0
\\
1 & 0
\\
1/2 & 0
\\
1/2 & 0
\end{matrix}
\right),
\end{equation*}
which leaves the first column in $\widehat{\bm{B}}$ unchanged, and has all zero entries in the second. We have
\begin{equation*}
\bm{B}_{1}^{-}\bm{1}
=
\left(
\begin{matrix}
3
\\
0
\end{matrix}
\right),
\qquad
\bm{B}_{1}^{-}\bm{1}\bm{1}^{T}\bm{B}
=
\left(
\begin{matrix}
3 & 3
\\
0 & 0
\end{matrix}
\right),
\qquad
\bm{I}\oplus\Delta^{-1}\bm{B}_{1}^{-}\bm{1}\bm{1}^{T}\bm{B}
=
\left(
\begin{matrix}
1 & 1
\\
0 & 1
\end{matrix}
\right),
\end{equation*}
and
\begin{equation*}
\bm{B}(\bm{I}\oplus\Delta^{-1}\bm{B}_{1}^{-}\bm{1}\bm{1}^{T}\bm{B})
=
\left(
\begin{matrix}
1/3 & 1/3
\\
1 & 1
\\
1/2 & 1/2
\\
1/2 & 1/2
\end{matrix}
\right).
\end{equation*}

As both columns in the last matrix coincide, we take one to write the least differentiating solution in the form
\begin{equation*}
\bm{x}=
\left(\begin{array}{cccc}
1/3
&
1
&
1/2
&
1/2
\end{array}
\right)^{T}
v,
\qquad
v
\ne
0.
\end{equation*}

Calculations with the other matrices obtained from $\widehat{\bm{B}}$ yield the same result, and are thus omitted.
\end{example}

To obtain the most differentiating score vectors we need to solve the problem
\begin{equation*}
\begin{aligned}
&
\text{maximize}
&&
\bm{1}^{T}\bm{B}\bm{u}(\bm{B}\bm{u})^{-}\bm{1}.
\end{aligned}
%\label{P-max1BuBu1}
\end{equation*}

Similarly as before, we reduce this problem to \eqref{P-maxqxAxp}, conclude that the conditions of Lemma~\ref{L-maxqxAxp} are fulfilled, and finally apply this lemma to obtain the next solution. 
\begin{theorem}
Let $\bm{B}=(\bm{b}_{j})$ be a matrix with columns $\bm{b}_{j}=(b_{ij})$, and $\bm{B}_{sk}$ denote the matrix obtained from $\bm{B}$ by fixing the entry $b_{sk}$ and replacing the others by $0$.

Then, the most differentiating score vectors are given by
\begin{equation*}
\bm{x}
=
\bm{B}(\bm{I}\oplus\bm{B}_{sk}^{-}\bm{B})\bm{v},
\quad
\bm{v}\ne\bm{0},
\quad
k
=
\arg\max_{j}\bm{1}^{T}\bm{b}_{j}\bm{b}_{j}^{-}\bm{1},
\quad
s
=
\arg\max_{i}b_{ik}^{-1}.
\end{equation*}
\end{theorem}

\begin{example}
We start with the solution at \eqref{E-xBu}, and compute 
$
\bm{1}^{T}\bm{b}_{1}
=
1$,
$
\bm{1}^{T}\bm{b}_{2}
=
1$,
$
\bm{b}_{1}^{-}\bm{1}
=
3$,
and
$
\bm{b}_{2}^{-}\bm{1}
=
6$.
Since
$
\bm{1}^{T}\bm{b}_{1}\bm{b}_{1}^{-}\bm{1}
=
3$
and
$
\bm{1}^{T}\bm{b}_{2}\bm{b}_{2}^{-}\bm{1}
=
6$,
we take
$
k=2$, $s=1$.

Next, we have
\begin{equation*}
\bm{B}_{12}
=
\left(
\begin{matrix}
0 & 1/6
\\
0 & 0
\\
0 & 0
\\
0 & 0
\end{matrix}
\right),
\quad
\bm{I}\oplus\bm{B}_{12}^{-}\bm{B}
=
\left(
\begin{matrix}
1 & 0
\\
2 & 1
\end{matrix}
\right),
\quad
\bm{B}(\bm{I}\oplus\bm{B}_{12}^{-}\bm{B})
=
\left(
\begin{matrix}
1/3 & 1/6
\\
2 & 1
\\
1 & 1/2
\\
1 & 1/2
\end{matrix}
\right).
\end{equation*}

Since the columns in the last matrix are collinear, we take one of them, say, the second, to write the most differentiating vector as
\begin{equation*}
\bm{x}
=
\left(
\begin{array}{cccc}
1/6
&
1
&
1/2
&
1/2
\end{array}
\right)^{T}
v,
\qquad
v\ne0.
\end{equation*}
\end{example}

\subsection*{Acknowledgements}
This work was supported in part by the Russian Foundation for Humanities (grant No. 16-02-00059). The author is very grateful to the referees for their valuable comments and suggestions, which have been incorporated into the revised version of the manuscript.

\bibliographystyle{abbrvurl}

\bibliography{Methods_of_tropical_optimization_in_rating_alternatives_based_on_pairwise_comparisons}

\end{document}